\documentclass[reqno]{amsart}
\usepackage{amsmath,amsfonts,amssymb,amsthm}
\usepackage{fancyhdr}
\DeclareMathOperator{\F}{Frob}

	\newtheorem{theorem}{Theorem} 
	\newtheorem*{theorem*}{Theorem} 
	
	\newtheorem*{lema*}{Lemma}
	\newtheorem{corollary}[theorem]{Corollary} 
	\newtheorem*{corollary*}{Corollary} 
	\theoremstyle{definition}				
	\newtheorem*{conjecture*}{Conjecture} 
	\newtheorem{proposition}[theorem]{Proposition}
	\newtheorem*{proposition*}{Proposition}
	\newtheorem*{problem*}{Problem}
	
	\newtheorem*{example*}{Example}
	\newtheorem{definition}[theorem]{Definition} 
	\newtheorem*{definition*}{Definition} 
	\newtheorem*{notation*}{Notation} 
	\newtheorem*{remark*}{Remark} 
	\newtheorem*{exercise*}{Exercise}

\title{A geometric approach to counting norms in cyclic extensions of function fields} 
\date{}
\author{
Vlad Matei
}
\address[Vlad Matei]{
\begin{itemize}
\item[-]
Department of Mathematics University of California Irvine, 340 Rowland Hall, Irvine, CA, 92697
\end{itemize}
}
 \email[Vlad Matei]{vmatei@math.uci.edu}

\begin{document}

\begin{abstract} In this paper we prove an explicit version of a function field analogue of a classical result of Odoni about norms in number fields \cite{odoni}, in the case of a cyclic Galois extensions. In the particular case of a quadratic extension, we recover the result Lior Bary-Soroker, Yotam Smilanski, and Adva Wolf in \cite{lior} which deals with finding asymptotics for a function field version on sums of two squares, improved upon  by Ofir Gorodetsky in \cite{ofir}, and reproved by the author in his Ph.D thesis using the method of this paper. The main tool is a twisted Grothendieck Lefschetz trace formula, inspired by the paper \cite{jordan}. Using  a combinatorial description of the cohomology we obtain a precise quantitative result which works in the $q^n\rightarrow \infty$ regime, and a new type of homological stability phenomena, which arises from the computation of certain inner products of representations. 
		\end{abstract}
		
		\maketitle

\section{Introduction}

 Odoni obtained in \cite{odoni}  an asymptotic result for integers that can be expressed  as norms elements of a number field $K$. In \cite{lior2}, a function field analogue of this problem was proved, namely the authors obtain the main term for any Galois extension of function fields. The goal of this paper is reprove this result for norms in the case of  the cyclic extension $\mathbb{F}_q[\sqrt[d]{T}]/ \mathbb{F}_q[T]$  and obtain an explicit polynomial in $q$ that gives this count.
 
 Let $\mathbb{F}_q$ be the finite field with $q$ elements where $q$ is an odd integer and suppose that $\mathbb{F}_q$ contains all $d$-th roots of unity or equivalently $q\equiv 1 \hspace{1mm} (\text{mod} \hspace{1mm} d)$. Let $\varepsilon$ be a primitive rooth of $d$-th order. We denote with $\mathcal{M}_{n,q}$ the set of monic polynomials of degree $n$ in $\mathbb{F}_q[T]$. We can describe norms in the cyclic extension $\mathbb{F}_q[\sqrt[d]{T}]/ \mathbb{F}_q[T]$ as follows:

\begin{definition*} Let $q$ be an odd prime power. For a polynomials $f\in \mathcal{M}_{n,q}$  we define the characteristic function:

  $$ b_{q}(f)=  \left\{
	\begin{array}{lrc}
		1,  & \mbox{if } f(T^d)=cg(T)g(\varepsilon T)\ldots g(\varepsilon^{d-1}T)   &\mbox{for} \hspace{3mm} g\in\mathbb{F}_q[T] \\
		0, & \mbox{otherwise.} 
	\end{array}
  \right. $$

where $c=\varepsilon^{nd(d-1)/2}$.

\end{definition*}

Consider the counting function  $\displaystyle S_q(n)=\sum_{\substack{f\in \mathcal{M}_{n,q}\\ f\hspace{1mm} \text{squarefree}}} b_q(f)$. What we shall prove is an expansion on the first theorem above for squarefree polynomials , namely

\begin{theorem} For every $n\geq 2$ and $q\equiv 1 \hspace{1mm} (\text{mod} \hspace{1mm} d)$ we can write  $\displaystyle  S_q(n)=\sum_{k=0}^{n} b_{k,n} q^{n-k}$  such that

a) $b_{k,n}=\displaystyle \sum_{j=k}^{2k} \delta_{k,j,n} \dbinom{1/d+(n-j-k-1)}{n-j-k}$;

b) We have that  $\delta_{k,j,n}=\delta_{k,j,n+1}$  for $n\geq 2k$ and 

$$ |\delta_{k,j,n}| \leq  C \dbinom{k}{j-k} (1.1)^{k}$$

for some absolute constant $C$.

\end{theorem}

Since there are exactly $q^{n-1}$ non-squarefree polynomials which are not norms we obtain the following:

\begin{corollary} We have that the number of polynomials that can be expressed as norms in the cyclic extension $\mathbb{F}_q[\sqrt[d]{T}]/ \mathbb{F}_q[T]$ is equal to $\dbinom{1/d+n-1}{n} q^n+O(q^{n-1})$.

\end{corollary}

If we remove the restriction on $f$ being squarefree, we can consider the following counting function $\displaystyle B_q(n)=\sum_{f\in \mathcal{M}_{n,q}} b_q(f)$. What we shall prove is the following, which is an explicit version of the result in \cite{lior2} for this particular Galois extension $\mathbb{F}_q[\sqrt[d]{T}]/ \mathbb{F}_q[T]$.

\begin{theorem}
For every $n\geq 2$ and $q\equiv 1 \hspace{1mm} (\text{mod} \hspace{1mm} d)$ we can write  $\displaystyle  B_q(n)=\sum_{k=0}^{n} c_{k,n} q^{n-k}$, where $c_{k,n}$ are independent of $q$, and they can be computed explicitly.

\end{theorem}

\begin{remark*} We will show this theorem in the last section, Section 7, and we will provide a recipe for computing each of the coefficients $c_{k,n}$.

\end{remark*}

The case $d=2$ is a function field analogue of a classical result of Landau. The problem over $\mathbb{Z}$ asks about the number of positive integers $\leq X$ that can be represented as $a^2+b^2$ for some nonnegative integers $a,b$. 
In \cite{lior1} the authors introduced the function field analogue of this problem, and obtain the first terms in this expansion, i.e the coefficients for $q^n$, $q^{n-1}$ and $q^{n-2}$. Ofir Gorodetski explored further this result \cite{ofir}, obtaining an asymptotic as $q^n\rightarrow \infty$. Moreover, using a generating  obtains a similar formula for coefficients of $q^j$ appearing above but using a generating function technique. 

The case $d=2$ was also the subject of the author's Phd thesis, and the methods used in proving Theorem 1 are the same. In the case $d=2$ we can use Theorem 1 to get a statement for all polynomials. Every polynomial can be factored uniquely as $f=A^2B$ and is a norm in the extension $\mathbb{F}_q[\sqrt{T}]/\mathbb{F}_q[T]$ if and  and only if $B$ is a norm; this fact was proven in \cite{lior}. Thus we can use this to  obtain a  recurrence relation in terms of squarefree polynomials, and combined  with Theorem 1 we obtain

\begin{corollary} For every $n\geq 2$ let $B_q(n)$ be the number of polynomials in $\mathcal{M}_{q,n}$ that can be written as $|U^2-TV^2|$ where $U,V\in\mathcal{M}_{q,n}$. 
Then  $\displaystyle  B_q(n)=\sum_{k=0}^{n} b_{k,n} q^{n-k}$  such that:

a) $b_{k,n}=\displaystyle \sum_{j=k}^{2k} \delta_{k,j,n} \frac{\dbinom{2(n-j)}{n-j}}{4^{n-j}}$;

b) We have that  $\delta_{k,j,n}=\delta_{k,j,n+1}$  for $n\geq 2k$ and 

$$ |\delta_{k,j,n}| \leq  C (1.1)^{k}$$

for some absolute constant $C$.

\end{corollary}

Finally we end some remarks regarding Theorem 1.

\begin{remark*}

$\bullet$ In the course of the proof of the theorem we shall give a geometric interpretation to the binomial coefficients appearing in the expansion, in terms of the Frobenius action on the roots of the polynomials arising in this statistic

$\bullet$ The stabilization of the coefficients  $\delta_{p,j,n}$ as $n$ gets large with respect to $p$ is explained by the stabilization of the multiplicity of a character paired against the cohomology of a certain space and thus can be  viewed thus a homological stabilization result. As remarked in \cite{kevin}, the class function which guides this statistic is not a character polynomial so it would be an interesting combinatorial problem to classify all types of possible stabilizations that could arise.

$\bullet$ In the case $d=2$ this statistic is the same as the number of polynomials that can be written as $|U^2-TV^2|$, but once $d\geq 3$ these statistics are not the same. A classical result of Erd\"{o}s and Mahler  states that the number of integers $\leq X$ that can be represented as $a^d+b^d$ for $d\geq 3$ is greater than $CX^{2/d}$. It would be interesting to prove an analog and even obtain a precise asymptotic over function fields.

\end{remark*}

Our strategy for the proof of Theorem 1 is a twisted Grothendieck Lefschetz formula inspired by work in  \cite{jordan}. This formula was independently in \cite{rowi} and generalized in \cite{kevin}. In a different direction a similar formula was proved for ramified coverings \cite{gadish}. 

\bigskip

\textbf{Acknowledgments} The author would like to thank  Jordan Ellenberg for suggesting this project, and helpful discussions on the subject matter of this paper.

\section{Preliminaries}

First we will restrict our attention to polynomials $f\in \mathcal{M}_{n,q}$ with $f(0)\neq 0$. This is due to nature of the representation theoretical result we shall use in our proof. So we if we denote with $\displaystyle S^{0}_q(n)=\sum_{\substack{f\in \mathcal{M}_{n,q}\\ f\hspace{1mm} \text{squarefree}, f(0)\neq 0}} b_q(f)$ it is easy to see that $S_q(n)=S^{0}_q(n)+S^{0}_q(n-1)$, since $T$ is a norm in the extension $\mathbb{F}_q[\sqrt[d]{T}]/ \mathbb{F}_q[T]$. 

For a squarefree polynomial $f\in\mathcal{M}_{n,q}$ with $f(0)\neq 0$ let the unordered $n$ tuple of  it's roots be $\{z_1,\ldots, z_n\}$ where $z_i\neq z_j$ and $z_i\neq 0$. Since $\F_q$ fixes the coefficients of the polynomial this induces a permutation on the roots of $f$. For each $z_i$ pick an $x_i\in\overline{\mathbb{F}}_q$ such that $x_i^d=z_i$.  It follows that  $\F_q$ induces a permutation on these representatives and moreover there is a signed component attached to each, corresponding  to a $d$-th root of unity. Call $R_{n,d}$ the space of all tuples $\{x_1,\ldots,x_n\}$  and  we shall also denote with $G_{n,d}=(\mathbb{Z}/d\mathbb{Z})^n\rtimes S_n$ the complex $d$-th roots of unity reflection group. 

 Now we can restate $b_q(f)=1$  if the $d$ roots of $x^d=z_i$ lie in different orbits of the $\F_q$ action on the space  $R_{n,d}$, thought  as the space of points on the tuples  $(x_i, \varepsilon x_i,\ldots \varepsilon^{d-1}x_i)$  where $\varepsilon$ is a primitive $d$ rooth of unity. These are the roots of $x^d=z_i$. More precisely  on each cycle of Frobenius, thought of as a permutation of the roots of $f$, we should have that the $d$-th roots of unity components corresponding to the $x_j$'s multiply to $1$.

\begin{definition*} Let $L_n$ be the subset of $G_{n,d}$, consisting of all $\mu_d$-signed permutation $\pi$ such that $x_i$ and $\varepsilon^{j}x_i$ lie in different orbits under $\pi$ for every $1\leq j\leq d$.

\end{definition*}

It is now time to relate the geometry of the space of roots and our counting problem. Note the fact  that $z_i\neq z_j$ imposes that $x_i\neq \varepsilon^{k} x_j$ for every $1\leq k\leq d$. 
It follows that we can identify $R_{n,d}$ as a hyperplane complement in affine $n$ space, namely $R_{n,d}=\{(\alpha_1,\ldots, \alpha_n)|\alpha_i\neq \varepsilon^{k}\alpha_j, \alpha_i\neq 0\}$.This is because for each $f\in \mathcal{M}_{n,q}$ considering the tuple $\{x_1,\ldots,x_n\}$ we can see this is a point of $R_{n,d}$ over $\overline{\mathbb{F}}_q$.  The representation theory and homological stability properties of this hyperplane arrangement are well understood; the interested reader can look at \cite{jen}.

\begin{definition} Let $\chi_n$ be the characteristic function  of $L_n$ as a subset of $G_{n,d}$.

\end{definition}
We shall prove a  theorem which relates the geometry of our space and the  counting problem, which is the same spirit as theorem 3.7 in \cite{jordan}. The main difference is to state and prove an analogous result for $G_{n,d}$  Galois covers instead of $S_n$ covers. 

To make it more explicit, if we consider a class function $\chi:G_{n,d}\rightarrow \mathbb{Q}$ then we can define its action on a squarefree polynomial $f$ in the following way: set $R(f)=\{z_1,z_2,\ldots,z_n\}$ to be sets of roots and we have an induced action of $\F_q$ on the set of $d$-th roots of these, namely $\{x_1,x_2,\ldots,x_n\}$ as above and this will give us a signed permutation $\sigma_f$.  We define $\chi(f)=\chi(\sigma_f)$ and we need to argue this is well defined. By forgetting the rooth of unity component on the  permutation, we recover the action of $\F_q$ on $R(f)$ and this has invariant cycle structure since cycles correspond to irreducible factors of $f$. Since conjugation preserves the cycle structure we are done.

\begin{theorem} Let $\displaystyle \mathcal{G}_q(n)=\sum_{f\in \mathcal{S}_{n,q},f(0)\neq 0} \chi(f)$. Then 

$$\mathcal{G}_q(n)=  \sum_{i} (-1)^i q^{n-i} \langle \chi, H^i_{\text{\'{e}t}}(R_{n,d};\mathbb{Q}_l)\rangle _{G_{n,d}}$$

\end{theorem}

 Here $\langle  \cdot , \cdot  \rangle $ denotes the standard product of class functions and the subscript $G_{n,d}$ denote the groups where the respective class functions live. 

As a corollary for our problem

\bigskip

\textbf{Corollary} We have that $$S^{0}_q(n)=  \sum_{i} (-1)^i q^{n-i} \langle \chi_n, H^i_{\text{\'{e}t}}(R_{n,d};\mathbb{Q}_l)\rangle _{G_{n,d}}.$$

\section{Proof of Theorem 4}

This is the main technical machinery to setup and prove for our problem. Let $\textbf{Conf}^{0}_{n}$ be the  affine complement $U_n=\{(z_1,z_2,\ldots, z_n)|z_i\neq z_j, z_i\neq 0\}$ modulo $S_n$. First let us argue that $R_{n,d}$ is a \'{e}tale Galois  cover of $\textbf{Conf}^{0}_n$ with Galois group $G_{n,d}$. 

Let $P_{n,d}$ be the set of monic degree $n$ polynomials which are norms in the extension $\mathbb{F}_q[\sqrt[d]{T}]/\mathbb{F}_q[T]$ . Consider the map 

$$\pi :\mathbb{A}^n\rightarrow P_{n,d}$$ defined by

$$\pi:(x_1,x_2,\ldots, x_n) \rightarrow f(T)= (T-x_1^d)(T-x_2^d)\ldots (T-x_n^d)$$ 
 
The map is well defined using proposition $22$. Note that the map is invariant under the $G_{n,d}$ action on the points $(x_1,x_2,\ldots, x_n)$ thus it factors through the scheme theoretic quotient $\mathbb{A}^n/G_{n,d}$ . Let us prove moreover that  the map
$\pi: \mathbb{A}^n/G_n\rightarrow P_n$ is an isomorphism. The $G_{n,d}$ invariant polynomials on $\mathbb{A}^n$ form a ring, namely $\mathbb{Z}[x_1,x_2,\ldots, x_n]^{G_{n,d}}$. First note that if such a polynomial is invariant under  the replacements $x_i\rightarrow \varepsilon^k x_i$'s then it has to be a polynomial in $x_i^d$'s. Thus $\mathbb{Z}[x_1,x_2,\ldots, x_n]^{G_{n,d}}=\mathbb{Z}[x_1^d,x_2^d,\ldots, x_d^2]^{S_n}$. As a function of $x_i$, the coefficient $a_i$ in $f$ is $\pm$ the $i$th symmetric polynomial $e_i(x_1^d,x_2^d,\ldots, x_n^d)$. The fundamental theorem of symmetric polynomials states

$$\mathbb{Z}[x_1^d,x_2^d,\ldots, x_n^d]^{S_n}=\mathbb{Z}[e_1,\ldots, e_n]=\mathbb{Z}[a_1,a_2,\ldots, a_n]$$,

thus giving the desired isomorphism. 

Under this map we can look thus at the preimage of $\textbf{Conf}^{0}_n$ and since this space can be identified with monic squarefree degree $n$ polynomials which do not vanish at zero, it can be easily seen that this preimage is  $R_{n,d}$. Since we can define $R_{n,d}$ in $\mathbb{A}^n$ as a hyperplane complement defined over $\mathbb{Z}[\mu_d]$, $R_{n,d}$ is a smooth $n$ dimensional scheme over $\mathbb{Z}[\mu_d]$. Since $\mu_d\subset \mathbb{F}_q$, the above scheme theoretic construction also works over our finite field. 

Since $G_{n,d}$ acts freely on $R_{n,d}$ by definition, restricting $\pi$ to a map $R_{n,d}\rightarrow \textbf{Conf}^{0}_n$ gives an \'{e}tale Galois cover with Galois group $G_{n,d}$. 

Now moving further note the fact that the Galois cover $R_n\rightarrow \textbf{Conf}^{0}_n$ gives a natural correspondence between finite-dimensional representations of $G_{n,d}$ and finite-dimensional local systems (locally constant sheaves) on $\textbf{Conf}^{0}_n$ that become trivial when restricted to $R_{n,d}$. Given $V$ a representation of $G_{n,d}$, let the $\chi_{V}$ be the associated character to it and let $\mathcal{V}$ be the corresponding local system on $\textbf{Conf}^{0}_n$. Initially this construction is done over $\mathbb{C}$ but since since every irreducible representation of $G_{n,d}$ can be defined over $\mathbb{Z}[\mu_d]$  (see \cite{geck},\cite{farb}), the local system $\mathcal{V}$ determines an $l$-adic sheaf and we shall not make a distinction between the two objects.

If $f(T) \in \textbf{Conf}^{0}_n$ and is a fixed point for the action of $\F_q$ on $\textbf{Conf}^{0}_n(\overline{\mathbb{F}}_q)$ then $\F_q$ acts on the stalk $\mathcal{V}_f$ over $f$. To give a concrete description, the roots of $f(T)$ are permuted by the action of Frobenius  on $\overline{\mathbb{F}}_q$, and moreover this induces a decorated permutation (each element comes with a choice of root of unity) on the $d$-th roots of the zeroes of the polynomial $f$, $\sigma_f$ which is defined up to conjugacy. The stalk $\mathcal{V}_f$ is isomorphic to $V$, and by choosing an appropriate basis the automorphism $\F_q$ acts according to $\sigma_f$. Thus we can conclude

$$ \text{tr}(\F_q:\mathcal{V}_f)=\chi_V(\sigma_f)  \hspace{4mm} (1)$$

The next ingredient we need is a version of the Grothendieck-Lefschetz trace formula with twisted coefficients. Namely for an appropiate system of coefficients $\mathcal{F}$, more precisely a $l$-adic sheaf, on a smooth projective variety $X$ defined  over $\mathbb{F}_q$ , we have :

$$\sum_{x\in X(\mathbb{F}_q)} \operatorname{tr}(\F_q|\mathcal{F}_x)=\sum_{i} (-1)^i \operatorname{tr}(\F_q:H^{i}_{\text{\'{e}t}}(X;\mathcal{F}))$$

This also holds for non-projective $X$, but we need to correct it by either using compactly supported cohomology or via Poincar\'{e} duality.

If we apply to our case using compactly supported cohomology we have that

$$ \sum_{f\in\textbf{Conf}^{0}_n(\mathbb{F}_q)} \text{tr}(\F_q|\mathcal{V}_f)= q^n\sum_{i} (-1)^i \text{tr}(\F_q:H^{i}_{c}(\textbf{Conf}^{0}_n;\mathcal{V}))  \hspace{4mm} (2)$$

Notice that the left hand side is exactly the statistical count on polynomials we need using $(1)$. The only thing left to unravel is the right hand side of the equality. 

First let us make some remarks about the setup. If $V$ is a $G_{n,d}$ representation, we denote by $\langle \chi, V\rangle$ the standard inner product of $\chi$ with the character of $V$; we can name this the multiplicity of $\chi$ in $V$, since this is true when $\chi$ is  irreducible, by Schur's lemma. Also note that for any class function on $G_{n,d}$ we can decompose it into a sum of irreducible characters and since both sides in  (2) are linear in $\chi$, it follows that we can reduce to the case of an irreducible character $\chi$ of $G_n$.

Let $\tilde{\mathcal{V}}$ denote the pullback of $\mathcal{V}$ to $R_{n,d}$.  Transfer gives us the isomorphism $H^{i}_{c}(\textbf{Conf}^{0}_n;\mathcal{V})\approx (H^i_{c}(R_{n,d};\tilde{\mathcal{V}}))^{G_{n,d}}$. Now we know that $\tilde{\mathcal{V}}$ is trivial on $R_{n,d}$, so we have 

$$  H^i_{c}(R_{n,d};\tilde{\mathcal{V}}) \approx H^i_{c}(R_{n,d};\mathbb{Q}_l) \otimes V $$ 

as $G_{n,d}$ representations. Putting it together 

$$H^{i}_{c}(\textbf{Conf}^{0}_n;\mathcal{V}) \approx (H^i_{c}(R_{n,d};\mathbb{Q}_l) \otimes V )^{G_n} \approx H^i_{c}(R_{n,d};\mathbb{Q}_l) \otimes_{\mathbb{Q}[G_{n,d}]} V$$

Now this gives the immediate consequence that $\text{dim}(H^{i}_{c}(\textbf{Conf}^{0}_n;\mathcal{V}))=\text{dim} ( H^i_{c}(R_{n,d};\mathbb{Q}_l) \otimes_{\mathbb{Q}[G_{n,d}]} V )$.

Since $V$ is self-dual as an $S_n$ representation,   $H^i_{c}(R_{n,d};\mathbb{Q}_l) \otimes_{\mathbb{Q}[G_{n,d}]} V$ is isomorphic to $\text{Hom}_{\mathbb{Q}[G_{n,d}]}(V; H^i_{c}(R_{n,d};\mathbb{Q}_l))$, whose dimension is computed using the inner product $\langle \chi, H^i_{c}(R_{n,d};\mathbb{Q}_l)\rangle$. 

Since $R_{n,d}$ is smooth of dimension $n$,  applying Poincar\'{e} duality gives 

$$ H^{2n-i}_{c}(R_{n,d};\mathbb{Q}_l)\approx \text{Hom} ( H^i_{\text{\'{e}t}}(R_{n,d};\mathbb{Q}_l); \mathbb{Q}_{l}(-n).$$

Since the action of $G_{n,d}$ on $\mathbb{Q}_{l}(-n)$ is trivial (this is the constant sheaf), we have that $\langle \chi, H^{2n-i}_{c}(R_{n,d};\mathbb{Q}_l)\rangle_{G_{n,d}}=\langle \chi, H^i_{\text{\'{e}t}}(R_{n,d};\mathbb{Q}_l)\rangle_{G_{n,d}}$. The last layer to uncover is the action of $\F_q$. 

We shall need the following theorem proved by Kim for a general field in \cite{kim} (see also Lehrer \cite{leh}).

\begin{theorem*}
Let $k$ be a field, and fix $l$ a prime different from the characteristic of $k$. Give a finite set of hyperplanes $H_1,\ldots, H_m$ in $\mathbb{A}^n$ defined over $k$, let $\mathcal{A}$ the complement: $\mathcal{A}=\mathbb{A}^n-\bigcup_{i=1}^{m} H_i$. Then:

\bigskip

(i) $H^{1}_{\text{\'{e}t}}(\mathcal{A};\mathbb{Q}_l)$  is spanned by the images of the $m$ maps: 

$$H^{1}_{\text{\'{e}t}}(\mathbb{A}^n-H_j;\mathbb{Q}_l) \rightarrow H^{1}_{\text{\'{e}t}}(\mathcal{A};\mathbb{Q}_l)$$ 

induced by the inclusion of $\mathcal{A}$ into $\mathbb{A}^n-H_j$ for $j=1,\ldots, m$.

\bigskip

(ii)  $H^{i}_{\text{\'{e}t}}(\mathcal{A};\mathbb{Q}_l)$ is generated by $H^{1}_{\text{\'{e}t}}(\mathcal{A};\mathbb{Q}_l)$ under cup product.

\end{theorem*}
 
As a consequence this theorem will give that the action of $\F_q$ on $H^i_{\text{\'{e}t}}(R_{n,d};\mathbb{Q}_l)$ is scalar multiplication by $q^i$. The action of $\F_q$ on $\mathbb{Q}_l(-n)$ is scalar multiplication by $q^n$ so the action of $\F_q$ on $H^{2n-i}_{c}(\textbf{Conf}^{0}_n;V))$ is scalar multiplication by $q^{n-i}$.

Putting it all together we obtain that

$$\text{tr}(\F_q: H^{2n-i}_{c}(\textbf{Conf}^{0}_n;\mathcal{V})))=q^{n-i} \langle \chi, H^i_{\text{\'{e}t}}(R_{n,d};\mathbb{Q}_l)\rangle_{G_{n,d}}$$

\section{Computation of the inner products}

To finish to proof of Theorem 2, note that according to Theorem 4 we just need to understand the inner product  $\langle \chi_n, H^i_{\text{\'{e}t}}(R_{n,d};\mathbb{Q}_l)\rangle _{G_{n,d}}$
 
Note that we can use instead  of the etale cohomology singular cohomology over $\mathbb{C}$, since for hyperplane arrangements, the cohomology depends only on the lattice of intersection of the hyperplane arrangement.

 To proceed to actual computations we shall need the following result in \cite{hend} which gives a  description of $H^i(R_n;\mathbb{C})$  as a $G_n$ representation. A similar result was obtained by \cite{doug}.

\begin{theorem*}[Henderson] As a representation of $G_n$, $ H^p(R_{n,d};\mathbb{C})$ is equal to $\displaystyle \bigoplus_{0\leq l \leq p} A^{l}(R_n)$ where $\displaystyle \varepsilon_n\otimes A^{l}(R_n)$ is isomorphic to the following direct sum:

$$\bigoplus_{\substack{\lambda^1,\lambda^2\\ |\lambda^1|+|\lambda^2|=n\\ l(\lambda^1)=n-p\\ l(\lambda^2)=l}} \operatorname{\huge{Ind}}^{G_n}_{\substack{(((\mu_d\times \mu_{\lambda^1_1})\times \ldots \times (\mu_d\times \mu_{\lambda^1_{n-p}})) \rtimes (S_{m_1(\lambda^1)}\times S_{m_2(\lambda^1)}\times \ldots)\\ \times((\mu_d\times \mu_{\lambda^2_1})\times \ldots \times (\mu_d\times \mu_{\lambda^2_{l}})) \rtimes (S_{m_1(\lambda^2)}\times S_{m_2(\lambda^2)}\times \ldots) ) }}  (\varepsilon \psi)$$

where, $\lambda^1=(\lambda^1_1,\ldots,\lambda^1_{n-p})$, $\lambda^2=(\lambda^2_1,\ldots,\lambda^2_l)$, $|\lambda^1|=\lambda^1_1+\ldots+\lambda^1_{n-p}$ and similarly for $\lambda^2$,  $\psi$ is the product of the standard inclusion characters $\mu_{\lambda_{a}^j}\hookrightarrow \mathbb{C}^{\times}$ and $\varepsilon$ is the product of the sign characters of the $S_{m_i(\lambda^1)}$ components.

\end{theorem*}

To make it easier to visualize the computations, we will interpret everything in term of matrix representations. The space $G_{n,d}$ can be thought of generalized permutation matrices where in each entry we replace the usual $1$ with an element of $\mu_d$. Now let's see how can we realize 

\begin{center}

$ H_{\lambda^1,\lambda^2}= ((\mu_d\times \mu_{\lambda^1_1})\times \ldots \times (\mu_d\times \mu_{\lambda^1_{n-p}})) \rtimes (S_{m_1(\lambda^1)}\times S_{m_2(\lambda^1)}\times \ldots)$

$\times((\mu_d\times \mu_{\lambda^2_1})\times \ldots \times (\mu_d\times \mu_{\lambda^2_{l}})) \rtimes (S_{m_1(\lambda^2)}\times S_{m_2(\lambda^2)}\times \ldots)  $

\end{center}

as a subgroup of $G_{n,d}$.

\begin{definition} A \emph{cell} is a  factor of the type $\mu_d\times \mu_v$.

\end{definition}

For constructing a matrix representative of the group  $\mu_d\times \mu_v=C_v$ note that we can take as  generators the $v\times v$ matrix $$\mathfrak{g}_v=\left\{ \begin{array}{lcr} g_{i+1,i}=1 \hspace{1mm} , \text{for} \hspace{1mm} 1\leq i\leq v \hspace{1mm} \text{where index is taken modulo} \hspace{1mm} v \\ 0\hspace{1mm}, \text{otherwise} \end{array}\right. $$ and it companions $\varepsilon g_v$ where $\varepsilon \in \mu_d$.

\begin{definition} A block is  a factor of the type  $(\mu_d\times \mu_v)^{m_v} \rtimes S_{m_v}$.

\end{definition}

To construct $H_{\lambda^1,\lambda^2}$ we first write the two partitions in descending order. Thus consider $\lambda_1=(a_1^{1},\ldots, a_r^{1})$ with $a_1^{1}>a_2^{1}>\ldots>a_r^{1}$ and each appears with multiplicity $m_i^{1}$ for $1\leq i \leq r$, and $\lambda_2=(a_1^{2},\ldots, a_s^{1})$ $a_1^{2}>a_2^{2}>\ldots>a_s^{2}$  and each appears  with multiplicity $m_j^{2}$ for $1\leq j\leq r$. Now we can construct the blocks $\mathcal{B}_i^{1}$ for $1\leq i\leq r$ and $\mathcal{B}_j^{2}$ for $1\leq j\leq s$. Finally we arrange the blocks on the main diagonal of the $n\times n$ matrix in order for $\lambda^1$ and $\lambda^2$,i.e, the diagonal will be $\mathcal{B}_1^1,\ldots, \mathcal{B}_r^{1}, \mathcal{B}_1^{2},\ldots, \mathcal{B}_r^{2}$. Each of the block from the above definition will be constructed as generalized permutation group on its \emph{cells}.

Now we proceed to the actual computation of the inner products  $\langle \chi_n, H^p(X_n;\mathbb{C})\rangle _{G_{n,d}}$. By theorem $5$ and Frobenius reciprocity it is equivalent to computing

$$\langle  \mbox{Res}^{G_{n,d}}_{H_{\lambda^1,\lambda^2}} \chi_n , \mbox{Res}^{G_{n,d}}_{H_{\lambda^1,\lambda^2}} \varepsilon_n\otimes  \varepsilon \psi \rangle$$, for each $\lambda^1$, $\lambda^2$ subject to the constraints in theorem 5. 

First  let us look closely at   $\mbox{Res}^{G_{n,d}}_{H_{\lambda^1,\lambda^2}} \varepsilon_n$.  Suppose the  the block structure of $H_{\lambda^1,\lambda^2}$ is given by the blocks $\mathcal{B}_1$, $\mathcal{B}_2$, $\ldots$, $\mathcal{B}_j$; here we ignore whether the blocks come from $\lambda^1$ or $\lambda^2$. Then noting that there is natural identification of $\varepsilon_n$ with the determinant of the corresponding permutation matrix

\begin{proposition} We have that $\varepsilon_n=\det(\mathcal{B}_1)\otimes\det(\mathcal{B}_2)\otimes \ldots \otimes \det(\mathcal{B}_j)$,

where $\det(\mathcal{B})$ is defined as the determinant of the block matrix $\mathcal{B}$ by ignoring the root of unity component $\mu_d$ in the entries.

\end{proposition}

We can restate this proposition also for our inner product. Namely, 

\begin{proposition}  

$$\langle \mbox{Res}^{G_{n,d}}_{H_{\lambda^1,\lambda^2}} \chi_n , \mbox{Res}^{G_n}_{H_{\lambda^1,\lambda^2}} \varepsilon_n\otimes  \varepsilon \psi \rangle= \langle \chi_{\mathcal{B}_1}, \mbox{det}\otimes (\varepsilon)_{\mathcal{B}_1} (\psi)_{\mathcal{B}_1} \rangle _{\mathcal{B}_1} \ldots  \langle \chi_{\mathcal{B}_j}, \mbox{det}\otimes (\varepsilon)_{\mathcal{B}_j} (\psi)_{\mathcal{B}_j} \rangle_{\mathcal{B}_j}$$

, where by abuse of notation we denote with $\chi_{\mathcal{B}}$ denotes the set of allowable signed permutations induced by the action of $\F_q$, $\varepsilon_{\mathcal{B}}$ is the restriction of $\varepsilon$ to the block, and similarly $\psi_{\mathcal{B}}$ is the restriction of $\psi$ to the block.

\end{proposition}

Further let us say a given block $\mathcal{B}$ is given by the factor of the type $(\mu_d\times \mu_v)^{m_v}\rtimes S_{m_v}$. Let us denote with $\mathcal{C}_1$, $\mathcal{C}_2$, $\ldots$, $\mathcal{C}_{m_v}$ the  \emph{cells} composing this block. Also we ignore the $\mu_d$ component on each \emph{cell}. Then we have

\begin{proposition} $\det(\mathcal{B})= \det(\mathcal{C}_1)\otimes\det(\mathcal{C}_2)\otimes \ldots \otimes \det(\mathcal{C}_{m_v})\otimes (\varepsilon_{m_v})^v$.

\end{proposition}

\begin{proof} We just have to note that to bring to diagonal form the block if we just think of \emph{cells} as a unit we would need an even or an odd number of moves to diagonalize according to the sign of $(\varepsilon_{m_v})$. Since \emph{cells} are $v\times v$ dimensional, to switch places of \emph{cells} requires $v$ moves.  Thus the total number of moves needed  to bring each \emph{cell} on the diagonal is multiplied by $v$ and thus it agrees with $(\varepsilon_{m_v})^v$. 
\end{proof}

By the above propositions we can work at the block level to compute the inner products.  Note that since blocks are arranged diagonally, for any matrix in $H_{\lambda^1,\lambda^2}$ the corresponding permutation on $n$ element has cycle structure determined by the cycle decomposition of each block; here we think of the block as generalized permutation but do not take into account the \emph{cell} structure it has. The cycles correspond to irreducible factors of the polynomial, since the permutation on $n$ elements encodes the $\F_q$ action on the roots. 
Since every element in the block, say it is $(\mu_d\times \mu_v)^{m_v}\rtimes S_{m_v}$, is a generalized permutation on its \emph{cells},  it's cycle decomposition, thought of a permutation in $S_{vm_v}$, is influenced by the $S_{m_v}$ component and the order of its \emph{cells} in $\mu_v$. We sum this up in the following proposition

\begin{proposition} Consider a block $(\mu_d\times \mu_v)^{m_v}\rtimes S_{m_v}$. Let $\sigma$ an element of $S_{m_v}$ and let $\mathcal{C}$ be an arbitrary cycle of $\sigma$. Ignoring the root of unity component $\mu_d$, let the order of the \emph{cells} on the cycle  be $a_{1},\ldots, a_{l(\mathcal{C})}$ modulo $v$. Then this arrangement will correspond to $\cfrac{v}{\gamma_v(a_1+...+a_{l(\mathcal{C})})}$ cycles of length $l(\mathcal{C})  \gamma_v(a_1+...+a_{l(\mathcal{C})})$ in the block structure, where we denote with $\gamma_v(x)$ denotes the order of the element in the additive group $\mathbb{Z}/v\mathbb{Z}$.
\end{proposition}




We can now proceed to the actual computations of the inner product. First we will consider

\begin{proposition} Suppose that $v\geq 2$. Consider the block $\mathcal{B}=(\mu_d\times \mu_v)\rtimes S_{m_v}$. We have that that the inner product  $\langle \chi_{\mathcal{B}}, \mbox{det}\otimes (\varepsilon)_{\mathcal{B}} (\psi)_{\mathcal{B}} \rangle _{\mathcal{B}}$ is equal to

$$ (-1)^{(v+i)m_v}\dbinom{(-1)^{v+i} h(v)+m_v-1}{m_v}$$

where $\mathcal{B}$ appears in $\lambda^i$, $\dbinom{\alpha}{n}=\cfrac{\alpha(\alpha-1)\ldots(\alpha-n+1)}{n!}$ is the generalized binomial function and $\displaystyle h(v)=\frac{1}{dv}\sum_{k=0}^{v-1}  \gcd(d,\gamma_v(k)) (-1)^{(v-1)k} \omega_v^{k}$, where $\omega_v$ is a primite root in $\mu_v$.

\end{proposition}

\begin{proof} We will denote with $\omega_r$ be a  primitive root of unity of order $r$ for any $r\in \mathbb{Z}$, $r$ positive. Note that  $\chi _B$ is multiplicative on the cyle decomposition of $\mathcal{B}$ as a permutation; here again we ignore all the $\mu_d$ components coming from the \emph{cells}. By the above this is influenced by the permutation arising from the $S_{m_v}$ component and the order of \emph{cells} as $\mu_v$ elements. Thus let $\sigma \in S_{m_v}$ and let $\mathcal{C}$ be a cycle in it's decomposition. 
We want on the big cycles of length $l(\mathcal{C})  \gamma_v(a_1+\ldots+a_{l(\mathcal{C})})$ for the $\mu_d$ component exponents to sum to a number divisible by $d$ for $\chi _B$ not to be automatically zero . Thus letting them be $x_i$, with $0\leq i\leq d-1$ what we need is that $d|\gamma_v(a_1+...+a_{l(\mathcal{C})})(x_1+\ldots+x_{l(\mathcal{C})})$. The number of such tuples can be easily computed; namely there are $d^{l(\mathcal{C})-1}\gcd(d, \gamma_v(a_1+...+a_{l(\mathcal{C})}))$. Next we see that the determinant of each \emph{cell} is equal to $(-1)^{(v-1)a_i}$. Using Proposition 11 and ignoring the $\varepsilon_{m_v}^v$ component- we will add it at the end of our computations, we conclude that the inner product on the cycle $\mathcal{C}$ is equal to

$$ d^{l(\mathcal{C})-1}\sum_{0\leq a_i\leq v-1} \gcd(d, \gamma_v(a_1+...+a_{l(\mathcal{C})})) (-1)^{(v-1)\sum a_j} \omega_{v}^{\sum a_j}.$$

For each $0\leq k\leq v-1$ the are $v^{l(\mathcal{C})-1}$ tuples $(a_1,a_2,\ldots,a_{l(\mathcal{C})})$ with sum equal to $k$ modulo $v$. Thus we can rewrite the above as

$$d^{l(\mathcal{C})-1}v^{l(\mathcal{C})-1} \sum_{k=0}^{v-1}  \gcd(d,\gamma_v(k)) (-1)^{(v-1)k} \omega_v^{k}.$$



We postpone simplifying $\displaystyle h(v)=\frac{1}{dv}\sum_{k=0}^{v-1}  \gcd(d,\gamma_v(k)) (-1)^{(v-1)k} \omega_v^{k}$ until the end of this section, since it will not affect our computations.
Remembering that for $\sigma$ we need to take the product over cycles of these inner products we computed for a specific cycle , we obtain that for a given permutation the total contribution from the inner product is $h(v)^{c(\sigma)}(dv)^{m_{v}}$, where $c(\sigma)$ is the number of cycles of the permutation $\sigma$. 

Further on we need to make a distinction between blocks appearing in $\lambda^1$ or $\lambda^2$; namely because the of $\varepsilon_{m_v}$ appearing only in $\lambda^1$. Note that for any permutation in $S_{m_v}$ we have  $\varepsilon(\sigma)=(-1)^{m_v-c(\sigma)}$. If our considered block appears in $\lambda^1$ we obtain the inner product sum is

$$\displaystyle (dv)^{m_v} \sum_{\sigma\in S_{m_v}}(-1)^{(v+1)(m_v-c(\sigma))} (h(v))^{c(\sigma)}$$

otherwise it appears in $\lambda^2$ so the inner product sum is 

$$\displaystyle (dv)^{m_v} \sum_{\sigma\in S_{m_v}}(-1)^{v(m_v-c(\sigma))} h(v)^{c(\sigma)}.$$

It is well known that $\displaystyle \sum_{\sigma \in S_n} X^{c(\sigma)}=X(X+1)...(X+n-1)$, see for example \cite{stanley}. Remembering that we need to average, thus divide by the size of the block as a group, we have that the inner products can be compactly written as:

$$ (-1)^{(v+i)m_v}\dbinom{(-1)^{v+i}h(v)+m_v-1}{m_v}$$

where $\mathcal{B}$ appears in $\lambda^i$ and $\dbinom{\alpha}{n}=\cfrac{\alpha(\alpha-1)\ldots(\alpha-n+1)}{n!}$ is the generalized binomial function.

\end{proof}

For $v=1$, looking at the block that contains $1$ we note that it is isomorphic to a generalized permutation group $G_{k,d}$. Thus the inner product at block level just simplifies to computing the proportion    $\cfrac{\# L_k}{\# G_{k,d}}$.

\begin{proposition} Suppose we have a block made of ones i.e $\mathcal{B}=(\mu_d)^{m}\rtimes S_m=G_m$. Then we have

$$\langle \chi_{\mathcal{B}},\mbox{det}\otimes (\varepsilon)_{\mathcal{B}} (\psi)_{\mathcal{B}} \rangle _{\mathcal{B}} =  \left\{ \begin{array}{lcr}

\dbinom{1/d+m-1}{m} & \mathcal{B} \in \lambda^1\\

(-1)^m\dbinom{-1/d+m-1}{m} & \mathcal{B} \in \lambda^2

	\end{array} \right.
$$
  
\end{proposition} 

\begin{proof}  We can repeat the same argument as in the previous proposition's proof, but it will be much simpler since our \emph{cells} have size $1$ so the $\psi$ and $\det$ of the cells components is trivial.

Thus for blocks of $1$ appearing in $\lambda^1$ since the $\varepsilon_m$ components cancel out and we just need to have the sum of orders of the $\mu_d$ components on each cycle be $0$ modulo $d$. Thus the inner product  is just

$$\frac{1}{d^m \cdot m!}\sum_{\sigma\in S_m} d^{m-c(\sigma)}=\dbinom{1/d+m-1}{m}$$

For the blocks of $1$ appearing in $\lambda^2$ we have

$$\frac{1}{d^m \cdot m!}\sum_{\sigma\in S_m} (-d)^m \cdot (-d)^{-c(\sigma)}=(-1)^m\dbinom{-1/d+m-1}{m}$$

\end{proof}

We now turn to simplifying our expression for the function $h(v)$ appearing in proposition $11$ and finding a precise criterion when it is nonvanishing. 

\begin{definition}

A positive integer $x$ is $d$-\emph{factorable} if all its primes factors are amongst the prime factors of $d$.

\end{definition}

\begin{proposition} Let $v\geq 2$ and factor  $v=2^{v_1}v_2$ with $v_2$ odd. Then if $v>2$ is $d$-\emph{factorable} we have

$$h(v) =  \left\{ \begin{array}{lcr}

\displaystyle \frac{1}{dv} \prod_{\substack{p|\gcd(d,v)\\ p\hspace{1mm} \text{prime}}} (1-p) &  v_1=0\\

\displaystyle \frac{2}{dv} \prod_{\substack{p|\gcd(d,v/2)\\ p\hspace{1mm} \text{prime}}} (1-p)-\frac{1}{dv}\prod_{\substack{p|\gcd(d,v)\\ p\hspace{1mm} \text{prime}}} (1-p)  &  v_1=1\\

	\displaystyle\frac{1}{dv} \prod_{\substack{p|\gcd(d,v/2)\\ p\hspace{1mm} \text{prime}}} (1-p) & v_1>1
	\end{array} \right.
$$

, $h(2)=\cfrac{1+\gcd(2,d)}{2d}$, and otherwise $h(v)=0$. We make the convention that the empty product is equal to zero. 

\end{proposition} 

\begin{proof}

We first condition on what possible values can $\gamma_v(k)$ take. Note that it is a divisor of $v$ and for each divisor $j$ of $v$ we have that each $k$ with $\gamma_v(k)=j$ can be written an $\cfrac{v}{j}$ times an interval element in $\mathbb{Z}/j\mathbb{Z}$. Moreover, it is well known that the sum of primitive roots of unity in $\mu_r$ is equal to $\mu(r)$ where, $\mu$ now is the Mobius function. 

If $v$ is odd then the we will simply have $\displaystyle \sum_{\gamma_v(k)=j} (-1)^{(v-1)k} \omega_v^{k}=\sum_{\gcd(s,j)=1} \omega_j^s=\mu(j)$. Thus we have that 

$$h(v)=\sum_{k=0}^{v-1}  \gcd(d,\gamma_v(k)) (-1)^{(v-1)k} \omega_v^{k}= \sum_{j|v} \gcd(d,j)\mu(j)$$

Otherwise $v$ is even, and we will need to split further according to $\cfrac{v}{j}$ even or odd. If $\cfrac{v}{j}$ is even then then like before $\displaystyle \sum_{\gamma_v(k)=j} (-1)^{(v-1)k} \omega_v^{k}=\sum_{\gcd(s,j)=1} \omega_j^s=\mu(j)$. If $\cfrac{v}{j}$ is odd then $\displaystyle \sum_{\gamma_v(k)=j} (-1)^{(v-1)k} \omega_v^{k}=\sum_{\gcd(s,j)=1} (-\omega_j)^s$. Note that $j$ is even so every $s$ in this sum is going to be odd, thus $\displaystyle \sum_{\gcd(s,j)=1} (-\omega_j)^s=-\mu(j)$. If $4|v$ then $4|j$ for any such divisor $j$, so $\mu(j)=0$. Thus we have that 
if $4\nmid v$

$$h(v)= \sum_{j|\frac{v}{2}} \gcd(d,j)\mu(j)-\sum_{j|\frac{v}{2}} \gcd(d,2j)\mu(2j)$$ 

  and 

$$h(v)= \sum_{j|\frac{v}{2}} \gcd(d,j)\mu(j)$$

if $4|v$.

To finish the proof of the propositon we have to see when $\displaystyle \sum_{j|u} \gcd(d,j)\mu(j)$ is nonzero for $u>1$. We can obviously factor our integer $u$ into $u_1u_2$ where $u_1$ has only prime factors from $d$, and $u_2$ shares no prime factors with $d$. Then we can also factor every divisor of $u$ as $j_1j_2$ where $j_1|u_1$ and $j_2|u_2$.  We can thus rewrite  

$$ \sum_{j|u} \gcd(d,j)\mu(j)=\sum_{j_1|u_1} \gcd(d,u_1)\mu(j_1)\sum_{j_2|u_2}\mu(j_2)$$

It is well known that for any positive integer $t>1$ we have $\sum_{s|t} \mu(t)=0$. Thus for the above to be nonzero we must have $u_2=1$ and we can further simplify it to the following product

$$\prod_{\substack{p|\gcd(d,u)\\ p\hspace{1mm} \text{prime}}} (1-p)$$

The two cases $v$ odd, and $4|v$ are easy to see that follow directly from the above. For $v>2$ even and $4\nmid v$, note that $\displaystyle \sum_{j|\frac{v}{2}} \gcd(d,j)\mu(j)+\sum_{j|\frac{v}{2}} \gcd(d,2j)\mu(2j)=\sum_{z|v} \gcd(d,z)\mu(z)$ so we can rewrite $h(v)=\displaystyle 2\sum_{j|\frac{v}{2}} \gcd(d,j)\mu(j)- \sum_{j|v} \gcd(d,j)\mu(j)$. This finishes the proof of the proposition.

\end{proof}

\begin{proposition} Suppose $ \langle  \mbox{Res}^{G_{n,d}}_{H_{\lambda^1,\lambda^2}} \chi_n , \mbox{Res}^{G_{n,d}}_{H_{\lambda^1,\lambda^2}} \varepsilon_n\otimes  \varepsilon \psi \rangle \neq 0$. Then
the parts appearing in each of the partitions  $\lambda^1$ and $\lambda$ should only be $1$'s, $2$'s, and  $d$-\emph{factorable} integers.

\end{proposition}

\begin{definition} A pair of partitions $(\lambda^1,\lambda^2)$ will be called acceptable if it satisfies the conditions of proposition 14.

\end{definition}

\section{Some explicit computations}

Before we proceed to the proof of theorem 2, let us first show how our previous computations, help us write down the contributions coming from the $H^0$, $H^1$ and $H^2$ to the inner products and thus narrow down the highest $3$ terms in our $q$-expansion for number of such polynomials; i.e find the coefficient of each of $q^n$, $q^{n-1}$ and $q^{n-2}$.
\begin{enumerate}

\item $H^0$. For $H^0$ from our description we only have $A^0$ and this is just the partitions $\lambda^1_{1}+\ldots+\lambda^{1}_n=n$ so that means $\lambda^1_{1}=\ldots=\lambda^{1}_{n}=1$. Thus we get the inner product to be 

$$\dbinom{1/d+n-1}{n}$$

\item $H^1$. We only have $A^0$ and $A^1$.

\indent $\bullet$ For $A^0$ we have $\lambda^1_{1}+\ldots+\lambda^{1}_{n-1}=n$ thus the only partition that works is $(2,1,\ldots,1)$. We obtain that the inner product is 

$$\frac{1+\gcd(2,d)}{2d} \dbinom{1/d+n-3}{n-2}$$

\indent $\bullet$ For $A^1$ we have $\lambda^1_{1}+\ldots+\lambda^{1}_{n-1}+\lambda^{2}_{1}=n$ so the only solution is $\lambda^1=(1,...,1)$ and $\lambda^2_{1}=1$. Thus the inner product is

$$\frac{1}{d}\dbinom{1/d+n-2}{n-1}$$

\item $H^2$. We have three parts $A^0$, $A^1$ and $A^2$.

\indent $\bullet$ For $A^0$ we have partitions $\lambda^1_{1}+\ldots+\lambda^{1}_{n-2}=n$. We need to split off in two cases when $3\nmid d$ or $3|d$. 
If $3\nmid d$ the only acceptable partition is $(2,2,1,\ldots,1)$. Thus the inner product is 

$$\dbinom{-h(2)+1}{2}\dbinom{1/d+n-5}{n-4}=-\frac{(\gcd(2,d)+1)(2d-1-\gcd(2,d))}{8d^2}\dbinom{1/d+n-5}{n-4}$$

If $3|d$ then we have an extra acceptable partition namely $(3,1,\ldots, 1)$. Thus the inner product is 

$$  -\frac{2}{3d}\dbinom{1/d+n-4}{n-3}-\frac{(\gcd(2,d)+1)(2d-1-\gcd(2,d))}{8d^2}\dbinom{1/d+n-5}{n-4}$$

\indent $\bullet$ For $A^1$ we have partitions  $\lambda^1_{1}+\ldots+\lambda^{1}_{n-2}+\lambda^{2}_{1}=n$ and we either have $\lambda^1=(2,1,\ldots,1)$ and $\lambda^2_{1}=1$ or $\lambda^1=(1,\ldots,1)$ and $\lambda^2_{1}=2$. Thus the inner product is

$$\frac{1+\gcd(2,d)}{2d^2} \dbinom{1/d+n-4}{n-3}+ \frac{1+\gcd(2,d)}{2d}\dbinom{1/d+n-3}{n-2}$$

\indent $\bullet$ For $A^2$ we have the relation  $\lambda^1_{1}+\ldots+\lambda^{1}_{n-2}+\lambda^{2}_{1}+\lambda^2_{2}=n$ and again the only solution is $\lambda^1=(1,\ldots,1)$ and $\lambda^2=(1,1)$. Thus the inner product is

$$-\frac{(d-1)}{2d^2} \dbinom{1/d+n-3}{n-2}$$

\end{enumerate}

To finish note again that the above gives the first three terms in the count for the count $S^{0}_q(n)$ and we have $S_q(n)=S^{0}_q(n)+S^{0}_q(n-1)$ thus the we need to add the above counts for $n$ and $n-1$. We obtain the following

\begin{theorem} Define $h_{d,n}=\dbinom{1/d+n-1}{n}$. The number of squarefree polynomials which are norms in the extension $\mathbb{F}_q[\sqrt[d]{T}]/\mathbb{F}_q[T]$ is

   $$S_q(n)=h_{d,n} q^n+ B_1q^{n-1}+B_2q^{n-2}+O(q^{n-3})$$

where $B_1=A_1h_{d,n-1}+A_2h_{d,n-2}$ and $B_2=A_3h_{d,n-2}+A_4h_{d,n-3}+ A_5h_{d,n-4}$ the coefficients $A_i$ have the following values:

$A_1=1-\cfrac{1}{d}$, $A_2=-\cfrac{1+\gcd(2,d)}{2d}$, $A_3=\cfrac{1+d\gcd(2,d)}{2d^2}$, $A_5=-\cfrac{(\gcd(2,d)+1)(2d-1-\gcd(2,d))}{8d^2}$ and 

$\bullet$ If $3|d$ we have $A_4=-\cfrac{2}{3d}+ \cfrac{1+\gcd(2,d)}{2d}$

$\bullet$ Otherwise $A_4= \cfrac{1+\gcd(2,d)}{2d}$.

\end{theorem}

\section{Proof of Theorem 1}

We can now proceed to the proof of theorem $1$. We will obtain bounds, but these will be far from optimal. Also we will not write an explicit formula for the coefficient of each $h_i$ term appearing; the previous two sections provide the recipe for computing out this coefficient. The combinatorics involved in simplifying further the expressions seems hard. 

We start  looking at the equality $\lambda^1_{1}+...+\lambda^1_{n-p}+\lambda^2_{1}+\ldots+\lambda^2_{l}=n$. We will group terms by  looking at the multiplicity of $1$ in $\lambda^1$. We will consider a fixed $l$ and afterwards will sum over the $l$'s. Let us first  characterize the possible values of multiplicity of $1$  in $(\lambda^1,\lambda^2)$.

\begin{proposition} The multiplicity of $1$ in $\lambda^1$, call it $a$,  satisfies $n-2p+l\leq a\leq n-p$.

\end{proposition}

\begin{proof} Let $a$ the multiplicity of $1$ in $\lambda^1$ and $b$ the multiplicity of $1$ in $\lambda^2$ in an acceptable pair $(\lambda^1,\lambda^2)$  with $\lambda^1_{1}+...+\lambda^1_{n-p}+\lambda^2_{1}+\ldots+\lambda^2_{l}=n$. Then since the other numbers appearing are at least equal to $2$ we have $a+b+2((n-p+l)-a+l-b)\leq n $ so simplifying yields the  bound $ a+b \geq n-2p+2l$. Now since $b\leq l$ it follows that $a\geq n-2p+l$. We obviously have $a\leq n-p$ so the proof is finished.
\end{proof}

 \bigskip

\textbf{Proof Theorem 1}

\bigskip

 By the previous proposition let $n-s-p$ be multiplicity of $1$ in the partition $\lambda^1$, where $0\leq s\leq p$. Thus we are left to write $s+p=\lambda^{1}_1+...+\lambda^{1}_{s}+\lambda^2_{1}+\ldots+\lambda^2_l$. 

We need to know in how many ways we can write an integer as a sum of $d$-factorable integers. This can be thought of as a  restricted partition question, and similar questions have appeared in the literature \cite{bru}, \cite{krenn}. The work in \cite{krenn} is a variant of our partition problem, where the authors restrict the possible digits in the writing a number as a sum  of $d$-factorable integers. Nonetheless their method can show that we have the following estimate:

\begin{proposition} Let $b_d(y)$ the number of partitions of  a positive integer $y$ into $d$-\emph{factorable} integers. Then there is an absolute constant $A$ such that 

$$b_d(y)<Ae^{\ln(y)^{m+1}}$$

where $m$ is the number of prime of factors of $d$.

\end{proposition}

\begin{remark*} We will use the above in the weaker form that $b^{*}_d(y)<A(1.1)^y$, where in the partitions associated $b^{\ast}_d$ we also allow to have $1$'s, $2'$, besides $d$-\emph{factorable} integers.

\end{remark*} 
 
To finish the estimate we note that for every partition of $p$ into $1$'s, $2$'s and $d$-\emph{factorable} integers we just have to pick the $s$ integers that go into the partition $\lambda^1$. 
Every such partition has at most $p$ term. If  we assume the contrary, note that in $\lambda^1$ we have $s$ parts at least equal to $2$, since we took the $1$'s out of $\lambda^1$ and  in $\lambda^2$ we would have at least $p-s$ parts greater than or equal to $1$. The total sum thus would be greater than $2s+p-s=p+s$ and this is a contradiction. Since every inner product is at most $1$, putting everything together, we obtain the crude upper bound for $\delta_{p,s,n}$:
$$|\delta_{p,p+s,n}|\leq A\dbinom{p}{s}(1.1)^p$$

The stabilization of $\delta_{p,s+p,n}$ is purely a combinatorial statement; namely for $n\geq 2p$ we can consider all possible acceptable partitions. This ends the proof of the theorem.

\hfill $\qed$

\section{Proof of Theorem 3}

\begin{proposition}(a) Every irreducible monic polynomial $P\in\mathbb{F}_q[T]$ has the property that $P(T^d)$ is either irreducible or  $P(T^d)$ completely factors as $cg(T)g(T\varepsilon)\ldots g(T\varepsilon^{d-1})$, where $g$ is also monic irreducible.

(b) Every polynomial $f\in\mathbb{F}_q[T]$ can be uniquely written as $A^dB$ where the multiplicites of irreducible factors in $B$ are less than $d-1$. The polynomial $f$  is a norm in the extension $\mathbb{F}_q[\sqrt[d]{T}]/\mathbb{F}_q[T]$ if and only if $B$ is a norm.

\end{proposition}

\begin{proof} (a) The statement is obvious for $P=T$. Thus let us suppose $P(T^d)$, $P\neq T$ is not irreducible and let $g(T)$ be one of it's irreducible factors.  Let $r_1,\ldots,r_n$ be the roots of $g$. Then obviously  $f$ has roots $r_1^d,r_2^d,\ldots, r_n^d$. This implies that $g(\varepsilon^{i}T)|P(T^d)$. Next we have that every two polynomials $g(T\varepsilon^i)$ and $g(T\varepsilon^j)$, for $i\neq j$ are coprime. Assuming the contrary, they would share a common root; say $r_s\varepsilon^i=r_t\varepsilon^j$. If $s\neq t$ it follows that $r_s^d=r_t^d$ and thus $P$ has a repeated root, which is not allowed for an irreducible polynomial. If $s=t$ then we must have $r_s=0$ and thus $T|P$. Since $P$ is irreducible we must have $P=T$, a contradiction.
We can conclude thus that $g(T\varepsilon^i)$ and $g(T\varepsilon^j)$ are coprime thus $g(T)g(T\varepsilon)\ldots g(T\varepsilon^{d-1})|P(T^d)$. Observing that $cg(T)g(T\varepsilon)\ldots g(T\varepsilon^{d-1})$  is a monic polynomial that can be written as $Q(T^d)$ for some $Q\in\mathbb{F}_q[T]$ it follows that $Q(T)|P(T)$ and since $P$ is irreducible we get $P=Q$.

(b)  Every polynomial $f\in\mathbb{F}_q[T]$ can be uniquely factored as $P_1^{a_1}P_2^{a_2}\ldots P_r^{a_r}$. Considering the exponents modulo $d$ we obtain the required writing. The backward implication is trivial so we  will just prove the forward implication. Let us look at the individual prime factors in our polynomial. If $P_i$ does not split in the extension $\mathbb{F}_q[\sqrt[d]{T}]/\mathbb{F}_q[T]$, and $f$ is a norm, i.e $f(T)=ch(T)h(T\varepsilon)\ldots h(T\varepsilon^{d-1})$, we must have that $P_i(T^d)|h(T\varepsilon^{j})$ since it is irreducible. Let $m_i$ be its multiplicity. We see then that also $P_i^{m_i}(T^d)|h(T\varepsilon^{k})$ for all $0\leq k\leq d-1$, and thus its multiplicity in the factorization of $f$ is $dm_i$, i.e, a multiple of $d$. Thus only split primes can appear in $B$ and the proof is finished using part (a).

\end{proof}

This leads to the following definition and corollary.

\begin{definition*} Let $S_{d,q}(n)$ be the number of  polynomials $P\in \mathcal{M}_{n,q}$ which are norms in the extension $\mathbb{F}_q[\sqrt[d]{T}]/\mathbb{F}_q[T]$ and such that the multiplicities of it's irreducible factors are at most $d-1$.

\end{definition*}

\begin{corollary} We have that

$$B_q(n)=\sum_{1\leq r\leq n, r\equiv n\hspace{1pt} (\text{mod} \hspace{1pt} d)} S_{d,q}(r) q^{(n-r)/d}$$

\end{corollary}

To finish the proof we will look at the generating function for the sequence $S_{d,q}(n)$. We done with $\mathcal{I}$ the set of monic, irreducible polynomials in $\mathbb{F}_q[T]$ which are norms in the extension  $\mathbb{F}_q[\sqrt[d]{T}]/\mathbb{F}_q[T]$. We have that

$$1+\sum_{n\geq 1}S_{d,q}(n)X^n=\prod_{f\in\mathcal{I}} \left(X^{\deg(f)(d-1)}+X^{\deg(f)(d-2)}+\ldots+X^{\deg(f)}+1\right)$$

Let us note that we've computed in Theorem $1$ the following generating function:

$$S(X)=1+\sum_{n\geq 1}S_{q}(n)X^n=\prod_{f\in\mathcal{I}}\left(X^{\deg(f)}+1\right)$$

To connect this with the above generating function, we just note that we have the following factorization:

$$X^{d-1}+X^{d-2}+\ldots+X+1=\prod_{i=1}^{d-1}\left(X(-\varepsilon^i)+1\right)$$ 

where $\varepsilon$ is a primitive $d$-th root of unity. This implies that  

$$1+\sum_{n\geq 1}S_{d,q}(n)X^n=S(-\varepsilon X)\ldots S(-\varepsilon^{d-1} X).$$

Thus we can compute $S_{d,q}(n)$ in terms of the sequence $S_q(n)$. We do not pursue writing this explicitly down, and this finishes the proof of theorem $3$.

\end{document}